\documentclass[12pt,pdflatex,reqno,a4paper]{amsart}

\usepackage{amsmath,amsfonts,amssymb}
\usepackage[foot]{amsaddr}

\usepackage[a4paper,lmargin=1.8cm,rmargin=1.8cm,tmargin=3.0cm,bmargin=3.0cm]{geometry}

\usepackage{tikz}
\usetikzlibrary{positioning,arrows,shapes,decorations.pathmorphing,shadows}

\usepackage[colorlinks,backref=page]{hyperref}

\newcommand{\U}[1]{\mathop{\mathrm{U}(#1)}}
\newcommand{\SO}[1]{\mathop{\mathrm{SO}(#1)}}
\newcommand{\Or}[1]{\mathop{\mathrm{O}(#1)}}
\newcommand{\GLR}[1]{\mathop{\mathrm{GL}(#1,\mathbb R)}}
\newcommand{\GLC}[1]{\mathop{\mathrm{GL}(#1,\mathbb C)}}
\newcommand{\GL}[1]{\mathop{\mathrm{GL}(#1)}}

\newcommand{\Gr}[3]{\mathop{\mathrm{Gr}_{#1}({#2}^{#3})}}
\newcommand{\Tau}[3]{\mathop{\mathrm{E}_{#1}({#2}^{#3})}}
\newcommand{\St}[3]{\mathop{\mathrm{V}_{#1}({#2}^{#3})}}
\newcommand{\GrR}[2]{\Gr{#1}{\RR}{#2}}
\newcommand{\TauR}[2]{\Tau{#1}{\RR}{#2}}
\newcommand{\GrC}[2]{\Gr{#1}{\CC}{#2}}
\newcommand{\TauC}[2]{\Tau{#1}{\CC}{#2}}
\newcommand{\StR}[2]{\St{#1}{\RR}{#2}}

\newcommand{\BG}[1]{\mathop{\mathrm{B}{#1}}}
\newcommand{\EG}[1]{\mathop{\mathrm{E}{#1}}}
\newcommand{\BGLR}[1]{\GrR{#1}{\infty}}
\newcommand{\EGLR}[1]{\TauR{#1}{\infty}}
\newcommand{\BGLC}[1]{\GrC{#1}{\infty}}
\newcommand{\EGLC}[1]{\TauC{#1}{\infty}}
\newcommand{\EO}[1]{\StR{#1}{\infty}}
\newcommand{\BGSO}[1]{\mathop{\widetilde{\mathrm{Gr}}_{#1}({\RR}^{\infty})}}
\newcommand{\VectR}[1]{\mathop{\mathrm{Vect}^{#1}_{\RR}}}
\newcommand{\VectC}[1]{\mathop{\mathrm{Vect}^{#1}_{\CC}}}
\newcommand{\rank}[1]{\mathop{\mathrm{rank}}(#1)}

\newcommand{\ug}{\;\shortstack{{\tiny\upshape def}\\=}\;}

\newcommand{\Id}{\mathop{\mathrm{Id}}}

\newcommand{\End}[1]{\mathop{\mathrm{End}(#1)}}
\newcommand{\ch}[1]{\mathop{\mathrm{ch}(#1)}}

\newcommand{\CC}{\mathbb{C}}
\newcommand{\HH}{\mathbb{H}}
\newcommand{\RR}{\mathbb{R}}
\newcommand{\ZZ}{\mathbb{Z}}
\newcommand{\NN}{\mathbb{N}}
\newcommand{\QQ}{\mathbb{Q}}
\newcommand{\OO}{\mathbb{O}}

\newcommand{\trivial}[1]{\varepsilon^{#1}}

\newcommand{\CP}[1]{\mathbb{C}P^{#1}}

\newcommand{\st}{\text{ such that }}
\renewcommand{\Im}{\mathop{\mathrm{Im}}}
\renewcommand{\Re}{\mathop{\mathrm{Re}}}

\numberwithin{equation}{section}

\newtheorem{thm}{Theorem}[section]
\newtheorem*{thm*}{Theorem}
\newtheorem{prp}[thm]{Proposition}

\newtheorem{lem}[thm]{Lemma}

\theoremstyle{definition}
\newtheorem{dfn}[thm]{Definition}

\theoremstyle{remark}
\newtheorem{rem}[thm]{Remark}

\usepackage{ifxetex}
\ifxetex
  \usepackage{fontspec}
\else
  \usepackage[T1]{fontenc}
  \usepackage[utf8]{inputenc}
  \usepackage{lmodern}
\fi

\title{Almost complex structures on spheres}
\date{\today}

\thanks{The second author were supported by the GNSAGA group of INdAM and by the MIUR under the PRIN Project ``Variet\`a reali e complesse: geometria, topologia e analisi armonica''}

\author{Panagiotis Konstantis}
\address{Philipps--Universit\"at Marburg, Fachbereich für Mathematik und Informatik, Hans-Meerwein-Stra\ss e, 35032 Marburg}
\email{pako@mathematik.uni-marburg.de}
\author{Maurizio Parton}
\address{Universit\`a di Chieti-Pescara\\ Dipartimento di Economia, viale della Pineta 4, I-65129 Pescara, Italy}
\email{parton@unich.it}

\begin{document}

\begin{abstract}
In this paper we review the well-known fact that the only spheres admitting an almost complex structure are $S^2$ and $S^6$. The proof described here uses characteristic classes and the Bott periodicity theorem in topological $K$-theory. This paper originates from the talk ``Almost Complex Structures on Spheres'' given by the second author at the MAM1 workshop ``(Non)-existence of complex structures on $S^6$'', held in Marburg from March 27th to March 30th, 2017. It is a review paper, and as such no result is intended to be original. We tried to produce a clear, motivated and as much as possible self-contained exposition.
\end{abstract}

\maketitle

\section{Introduction}

A \emph{complex manifold} is a differentiable manifold $M$ whose transition functions are holomorphic. If $x+iy$ denotes local coordinates around $p\in M$, the multiplication by $i$ induces an endomorphism $J_p$ of the tangent space $T_p(M)$ given by
\begin{equation}
\label{eq:cpx_mfd}
\left\{
\begin{aligned}
\frac{\partial}{\partial_x}|_p&\longmapsto\frac{\partial}{\partial_y}|_p\\
\frac{\partial}{\partial_y}|_p&\longmapsto-\frac{\partial}{\partial_x}|_p
\end{aligned}
\right.
\end{equation}
Writing the Cauchy-Riemann equations on the intersection of two overlapping charts, one obtains that $J_p$ is globally defined. The endomorphism $J:T(M)\rightarrow T(M)$, locally defined by \eqref{eq:cpx_mfd}, is an example of almost complex structure. More generally, an \emph{almost complex structure} on a real differentiable manifold $M$ is a linear bundle morphism $J$ of the tangent vector bundle $T(M)$ satisfying $J^2=-\Id$. The pair $(M,J)$ is called an \emph{almost complex manifold}, and whenever $J$ arises from \eqref{eq:cpx_mfd} for certain holomorphic coordinates $x+iy$, one says that the almost complex structure $J$ is \emph{integrable}.

One often refers to the above discussion by the sentence: ``on a complex manifold the multiplication by $i$ gives a compatible integrable almost complex structure''.

The problem of deciding if a given almost complex structure $J$ is integrable is nowadays an easy task, due to the celebrated Newlander-Nirenberg Theorem \cite{NeNCAC}. Much more challenging is the problem of deciding whether a differentiable manifold admits \emph{any} integrable almost complex structure: in real dimension $>4$, the only known general obstruction is the trivial one, that is, the existence of an almost complex structure.

The existence of an almost complex structure $J$ on a manifold $M$ is possible only if the dimension $m$ of $M$ is even, because of $(\det J)^2=(-1)^m$. Moreover, $J$ is equivalent to a $\GLC{m/2}$-structure on $T(M)$. In fact, any endomorphism $J_p$ on the tangent space $T_p(M)$ such that $J_p^2=-\Id$ admits an \emph{adapted} basis $\{v,J(v)\}$, and with respect to this basis
$J_p=
\begin{pmatrix}
0,-1\\
+1,0
\end{pmatrix}$. Thus, all the $J_p$ stick together giving an almost complex structure $J$ if and only if the structure group $\GLR{2n}$ can be reduced to $\GLC{n}$. See \cite[Remark 2 at page 8]{BrySRG} for a nice discussion regarding the terminology ``almost'' and ``integrable'' for $G$-structures.

Very clear references for basic facts on almost complex structures are the classical Kobayashi-Nomizu \cite[IX.1 and IX.2]{KoNFD2} and the Besse \cite[2.A]{BesEiM}.

This paper originates from the talk ``Almost Complex Structures on Spheres'' given by the second author at the \href{http://www.mathematik.uni-marburg.de/~agricola/Hopf2017/}{MAM1 workshop ``(Non)-existence of complex structures on $S^6$''}, held in Marburg from March 27th to March 30th, 2017. The plan of this beautiful workshop was to discuss and summarize various results of the past 30 years about the existence or non-existence of integrable almost complex structures on spheres, and this talk was about the classical result of Borel and Serre ruling out all spheres besides $S^2$ and $S^6$, namely, the result from where the story began:
\begin{thm}[Borel and Serre, 1953 \cite{BoSGLP}]\label{thm:main}
The sphere $S^n$ admits an almost complex structure if and only if $n=2$ or $n=6$.
\end{thm}

The material in this paper comes from several sources, clearly stated whenever necessary. As such, no result is intended to be original. However, we tried to produce a clear, motivated and as much as possible self-contained exposition.




\section{Almost complex structures on $S^2$ and $S^6$}

An almost complex structure on $S^2$ can be constructed in several ways. One can observe that since $S^2$ is orientable, its structure group $\GLR{2}$ can be reduced to $\SO{2}=\U{1}\subset\GLC{1}$. Or, one can observe that $S^2$ is diffeomorphic to $\CP{1}$.

An explicit construction of $J$ is possible by means of a vector product in $R^3$. The feasibility of this approach, that we are now going to describe, is that the same construction leads to an almost complex structure on $S^6$.

Denote by $\HH$ and $\OO$ the quaternions and the octonions, respectively. We will denote the product of $u,v\in\HH$ or $u,v\in\OO$ by juxtaposition $uv$, and the conjugation by $u^*$. Note then that $(uv)^*=v^*u^*$, and $uu^*=|u|^2$, where $|\cdot|$ is the standard Euclidean norm. By polarization, we get $2\langle u,v\rangle=uv^*+v^*u$, which implies the useful identity $uv^*=-v^*u$ whenever $u\perp v$. On the octonions, we have a non-trivial \emph{associator}, that is, $[u,v,w]\ug((uv)w)-(u(vw))$ is not identically zero if $u,v,w\in\OO$. A key property of this associator is that it is \emph{alternating}, that is, $[u,v,w]=0$ whenever two amongst $u,v,w$ are equal.

Look at the spheres $S^2$ and $S^6$ as unit spheres in $\Im\HH$, $\Im\OO$, respectively:
\[
S^2=\{x\in\Im\HH\st |x|^2=1\},\qquad S^6=\{x\in\Im\OO\st |x|^2=1\}
\]

Let $V$ denote either $\Im\HH$ or $\Im\OO$. A cross product $u\times v$ can be defined on $V$ by:
\[
u\times v\ug\Im(uv)=\frac{1}{2}(uv-vu)\qquad u,v\in V
\]

\begin{prp}
Let $p\in S^n\subset V$, $n=2$ or $n=6$. Define $J\in\End{T(S^n)}$ by $J_p(v)\ug b\times v$, where $v\in T_p(S^n)\subset V$. Then $J$ is an almost complex structure on $S^n$, that is, $J_p^2=-\Id$ for any $p\in S^n$.
\end{prp}

\begin{proof}
First, we prove that if $u\perp v$, then $u\times v=uv$:
\[
(uv)^*=v^*u^*\stackrel{v\in\Im}{=}-vu^*\stackrel{u\perp v}{=}uv^*=-uv\implies\Re(uv)=0\implies u\times v=\Im(uv)=uv
\]
Now, since $p\perp v$, we have $J_p^2(v)=p(pv)\stackrel{[p,p,v]=0}{=}(pp)v=-(pp^*)v=-(|p|^2)v=-v$.
\end{proof}

\begin{rem}
The algebras $\HH$ and $\OO$ belong to an infinite family of algebras $A_n$ of dimension $2^n$ arising from the \emph{Cayley-Dickson process}. One could wonder whether the above construction works for any other $A_n$, apart from $\HH=A_2$ and $\OO=A_3$. The answer is no, essentially because in order to define a cross product, you need $A_n$ be alternative, and this happens only for $n$ up to 3. To deepen this part of the story, see \cite[Appendix IV.A, page 140 and forward]{HaLCaG}, \cite[Hurwitz 1898 paper, in German]{HurNGW}, \cite[Section 2.2]{BaeOct}, \cite{BrGVCP}, and the recent \cite{MclWDC}.
\end{rem}

\begin{rem}
Denote by $N_J$ the Nijenhuis tensor of $J$. Then, the almost complex structure on $S^2$ is integrable, because it is the complex structure induced by $\CP{1}$ (this can be seen geometrically, because $p\times v$ is a $\pi/2$ counterclockwise rotation of $v$ on the plane $p^\perp$, or observing that in complex dimension one $N_J$ is 0, thus $(S^2,J)$ is a compact complex curve of genus 0, the only one being $\CP{1}$), whereas the almost complex structure on $S^6$ is not integrable \cite{EcFISP}. The following statement is then true \emph{a fortiori}:
\[
N_J=0\iff [u,v,w]=0\qquad\forall u,v,w\in V
\]
This suggests a direct relation between $N_J(u,v)$ and some expression involving associators. A natural one would be $\langle N_J(u,v),w\rangle=[u,v,w]$. However, we have not been able so far to find such a relation.
\end{rem}

\section{Characteristic classes}

According to Dieudonné \cite[IV.1]{DieHAD}, characteristic classes arose in 1935 in Stiefel's dissertation \cite{StiRFN}, as a tool to investigate whether a manifold $M^n$ admits $m\le n$ linearly independent nowhere vanishing vector fields. Looking at the $(r+1)$-skeleton of $M$, after orthonormalizing, such $m$ vector fields define maps $\alpha$ from disks $D^{r+1}$ into the Stiefel manifold $\StR{m}{n}$ of orthonormal $m$-frames in $R^n$. Since this argument reduces the problem to whether these $m$ vector fields can be extended from the sphere $S^r$ to the disk $D^{r+1}$, the problem of existence can be read in terms of homotopy groups $\pi_r(\StR{m}{n})$, for $r\ge n$. Since $\pi_{r}(\StR{m}{n})=0$ for $r<n-m$, such $m$ vector fields always exist on the $(m-n)$-skeleton of $M$, and the first obstruction appears as $\alpha\in\pi_{n-m}(\StR{m}{n})$. This $\alpha$ is what was called the \emph{characteristic} of the $m$-frame by Stiefel, giving birth to the name \emph{characteristic classes}.

\subsection{Classifying spaces}

The above argument is the obstruction-theoretic definition of \emph{Stiefel-Whitney} characteristic classes of the tangent bundle of a manifold $M$. More generally, characteristic classes are associated to a (real or complex) vector bundle $E\rightarrow B$: they are defined as cohomology classes of $B$ induced by any classifying map $\phi:B\rightarrow \BG{\GL{n}}$, where $\BG{\GL{n}}=\BG{\GLR{n}}$ or $\BG{\GLC{n}}$ is the classifing space for real or complex vector bundles.

In the following, we sketch the notions of classifying maps, classifying spaces and universal bundles in the general setting of principal $G$-bundles. For a historical description going from 1935 Whitney paper \cite{WhiSpS} to 1956 Milnor beautiful construction \cite{MilCUB,MilCII}, see \cite[III.2E-2G]{DieHAD}. The main reference for this part is \cite{HatVBK}.

Let $G$ be a topological group, acting freely on a \emph{weakly contractible} topological space $\EG{G}$ (that is, with trivial homotopy groups). The corresponding principal $G$-bundle $\pi_G:\EG{G}\rightarrow\BG{G}$ is called a \emph{universal $G$-bundle}, and the base space $\BG{G}=\EG{G}/G$ a \emph{classifying space} of $G$. The reason for the name ``universal bundle'' is the following theorem.

\begin{thm*}
Any principal $G$-bundle $\pi:E\rightarrow B$ over a CW-complex $B$ is isomorphic to a pull back of a universal $G$-bundle under some continuous map $f:B\rightarrow\BG{G}$. Moreover, two pull backs are isomorphic if and only if the corresponding maps are homotopic.
\end{thm*}

The map $f$ in the theorem is called a \emph{classifying map} for the principal bundle $E\rightarrow B$. For $B$ a CW-complex, the reason why the universal bundle exists is the \emph{Brown's representability theorem}, a result in Category Theory \cite{BroCoT}. However, hypothesis can be relaxed in several directions: for instance, $B$ can be a paracompact topological space instead of a CW-complex.

The need for a universal bundle arise from the classification of fiber bundles with a fixed $G$ as structure group. After reducing to the associated principal $G$-bundle $E\rightarrow B$ one can, heuristically, look for a principal $G$-bundle $\EG{G}\rightarrow\BG{G}$ \emph{depending only on $G$} such that \emph{any principal $G$-bundle} is induced by $\EG{G}\rightarrow\BG{G}$ by means of a bundle map $\phi$:
\begin{center}
\begin{equation}\label{dgr:univ}
\begin{tikzpicture}[scale=1]
\node (A) at (0,1) {$E=\phi^*(\EG{G})$};
\node (B) at (2,1) {$\EG{G}$};
\node (C) at (0,0) {$B$};
\node (D) at (2,0) {$\BG{G}$};
\path[->]
(A) edge (B)
(C) edge node[above]{$\phi$} (D)
(A) edge (C)
(B) edge (D);
\end{tikzpicture}
\end{equation}
\end{center}

Using the left and the right action of $G$ respectively on $E$ and $\EG{G}$, and the canonical projections, one obtains the following diagram:
\begin{center}
\begin{tikzpicture}[scale=1]
\node (A) at (0,1) {$E=\phi^*(\EG{G})$};
\node (B) at (5,1) {$\EG{G}$};
\node (C) at (0,0) {$B$};
\node (D) at (5,0) {$\BG{G}$};
\node (E) at (3,1) {$E\times\EG{G}$};
\node (F) at (3,0) {$E\times_G\BG{G}$};
\path[->]
(E) edge (A)
(E) edge (B)
(E) edge (F)
(F) edge node[above]{$\pi$} (C)
(F) edge node[above]{$p$} (D)
(A) edge (C)
(B) edge (D);
\end{tikzpicture}
\end{center}
Thus, if one can find a section $s$ of $\pi$, the classifying map is $\phi=p\circ s$: but such a section $s$ exists whenever all the homotopy groups of $\EG{G}$ vanish - here is where the need for $\EG{G}$ to be weakly contractible comes from. As said before, under mild hypothesis on the base space $B$ and the group $G$, such a weakly contractible space $\EG{G}$ always exists.
Moreover, since everything behaves well with respect to homotopies, one obtains a bijective correspondence between principal bundles and homotopy classes of continuous maps:
\[
\{\text{isomorphism classes of principal $G$-bundles over $B$}\}\stackrel{1-1}{\longleftrightarrow} [B:\BG{G}]
\]

Concrete examples of classifying spaces can be given using the notion of Grassmannian and Stiefel manifold. For any $n\le k$, denote by $\GrR{n}{k}$ the set of $n$-dimensional vector subspaces of $\RR^k$, and consider the set
\[
\TauR{n}{k}\ug\{(v,l)\in\RR^k\times\GrR{n}{k}\st v\in l\}
\]
The map sending $(v,l)\in\TauR{n}{k}$ into $l\in\GrR{n}{k}$ defines a \emph{tautological} vector bundle $\TauR{n}{k}\rightarrow\GrR{n}{k}$. Because of Gram–Schmidt, its associated $\GLR{n}$-principal bundle and $\Or{n}$-principal bundle, having as fiber over $l\in\GrR{n}{k}$ respectively the isomorphisms and the isometries from $\RR^n$ to $l$, have the same homotopy type. The latter is called the \emph{Stiefel manifold} $\StR{n}{k}$ of ordered orthonormal $n$-frames in $\RR^k$. Using the fibration $\StR{n-1}{k-1}\hookrightarrow\StR{n}{k}\twoheadrightarrow S^{k-1}$, one sees that $\StR{n}{k}$ is compact, and that the first non-trivial homotopy group of the Stiefel manifold is $\pi_{k-n}(\StR{n}{k})$. Considering the sequence of inclusions $\RR\subset\RR^2\subset\RR^3\subset\dots$, one can define the direct limit versions of the above: the tautological bundle $\EGLR{n}\rightarrow\BGLR{n}$ over the infinite Grassmannian and the associated infinite Stiefel manifold $\EO{n}$. Note that $\EO{n}$ is no more compact, but is weakly contractible.

Consider the tangent bundle $T(S^n)\rightarrow S^n$ of the sphere $S^n\subset\RR^{n+1}$. For any $x\in S^n$, we have $T_x(S^n)\subset\RR^{n+1}$, and this gives a map $\phi:S^n\rightarrow\GrR{n}{n+1}$. The fiber on $x\in S^n$ are the vectors that are tangent to the sphere in $x$. Read in $\RR^{n+1}$, these vectors are the elements of $\phi(x)$, that is, the fiber over $x\in S^n$ is ``tautologically'' the fiber of $\TauR{n}{n+1}\rightarrow\GrR{n}{n+1}$. In this way, diagram \eqref{dgr:univ} becomes
\begin{center}
\begin{tikzpicture}[scale=1]
\node (A) at (0,1) {$T(S^n)$};
\node (B) at (3,1) {$\TauR{n}{n+1}$};
\node (C) at (0,0) {$S^n$};
\node (D) at (3,0) {$\GrR{n}{n+1}$};
\path[->]
(A) edge (B)
(C) edge (D)
(A) edge (C)
(B) edge (D);
\end{tikzpicture}
\end{center}

The above argument works whenever $B$ is a smooth manifold, because of the Whitney embedding theorem. More generally, one can show that every rank $n$ vector bundle $E\rightarrow B$ over a paracompact topological space $B$ is induced by $\EGLR{n}\rightarrow\BGLR{n}$. Since the total space of the associated $\GLR{n}$-principal bundle has the same homotopy type as $\EO{n}$, and $\EO{n}$ is weakly contractible, we have that $\BGLR{n}$ is at the same time the classifying space $\BG{\GLR{n}}$ for real vector bundles and the classifying space $\BG{\Or{n}}$ for real Euclidean vector bundles. A similar argument shows that $\BG{\GLC{n}}=\BG{\U{n}}=\BGLC{n}$ is the classifying space for complex vector bundles, and complex Hermitian bundles as well. In summary:
\begin{align}
\VectR{n}{B}&=[B:\BGLR{n}]\label{eq:vectreal}\\
\VectC{n}{B}&=[B:\BGLC{n}]\label{eq:vectcpx}
\end{align}

\subsection{Stiefel-Whitney and Chern classes}


In theory, formulas \eqref{eq:vectreal} and \eqref{eq:vectcpx} solve the classification problem for vector bundles. In real life, the set of homotopy classes of maps is very difficult to compute. However, the notion of universal bundle $\EG{G}\rightarrow\BG{G}$ can be used to define $G$-bundle invariants.

\begin{dfn}\label{dfn:classes}
Let $E\rightarrow B$ be a $G$-principal bundle, and $f:B\rightarrow\BG{G}$ a classifying map. Cohomology classes $c$ of $B$ induced by cohomology classes $c'$ of $\BG{G}$, namely $c=f^*(c')\in H^*(B;R)$, are called \emph{characteristic classes} of $E\rightarrow B$. Here $R$ is any ring of coefficients.
\end{dfn}

This definition works because different classifying maps are homotopic, and two homotopic maps induce the same homomorphism on cohomology. Characteristic classes of vector bundles are then the pull-back by any classifying map of cohomology classes of the Grassmannian. In the following theorem, we describe the cohomology ring structure of $\BGLR{n}$ and $\BGLC{n}$ with coefficients in $\ZZ$ and $\ZZ_2$.

\begin{thm}\label{thm:cohomology}
\begin{align}
H^*(\BGLR{n};\ZZ_2)&=\ZZ_2[w_1,\dots,w_n],\qquad w_i\in H^i(\BGLR{n};\ZZ_2)\label{frm:st-wh}\\
H^*(\BGLC{n};\ZZ)&=\ZZ[c_1,\dots,c_n],\qquad c_i\in H^{2i}(\BGLC{n};\ZZ)\label{frm:chern}
\end{align}
\end{thm}

Theorem \ref{thm:cohomology} can be proved in several ways. One can use the sphere bundles $S(\EGLR{n})\rightarrow\BGLR{n}$, $S(\EGLC{n})\rightarrow\BGLC{n}$ and their associated Gysin sequences. Or one can use the \emph{splitting principle} \cite[from page 273]{BoTDFA}, as for instance in \cite[Theorem 7.1]{MiSChC}. A third possibility is using cellular cohomology, and a cellular decomposition of the Grassmannian by means of \emph{Schubert cells}, see for instance the very clear description in \cite[pages 31--34]{HatVBK}.

Combining definition \ref{dfn:classes} and theorem \ref{thm:cohomology}, we define the Stiefel-Whitney and the Chern classes. For Stiefel-Whitney classes, this approach was used for the first time by Pontrjagin in \cite{PonCCM,PonCCD}, see \cite[IV.3.C]{DieHAD}.

\begin{dfn}\label{dfn:st-wh}
Let $E\rightarrow B$ be a rank $n$ real vector bundle, $f:B\rightarrow\BGLR{n}$ a classifying map and $w_i$ the generator of $H^i(\BGLR{n};\ZZ_2)$ as in formula \eqref{frm:st-wh}. The cohomology classes $w_i(E)\ug f^*(w_i)\in H^i(B;\ZZ_2)$, for $i=1,\dots,n$, are called the \emph{Stiefel-Whitney classes} of $E\rightarrow B$. The \emph{total Stiefel-Whitney class} is then $w(E)\ug 1+w_1(E)+\dots+w_n(E)$.
\end{dfn}

\begin{dfn}\label{dfn:chern}
Let $E\rightarrow B$ be a rank $n$ complex vector bundle, $f:B\rightarrow\BGLR{n}$ a classifying map and $c_i$ the generator of $H^{2i}(\BGLC{n};\ZZ)$ as in formula \eqref{frm:chern}. The cohomology classes $c_i(E)\ug f^*(c_i)\in H^{2i}(B;\ZZ)$, for $i=1,\dots,n$, are called the \emph{Chern classes} of $E\rightarrow B$. The \emph{total Chern class} is then $c(E)\ug 1+c_1(E)+\dots+c_n(E)$.
\end{dfn}

\begin{rem}\label{rem:axiom}
Denote by $\cdot$ the product in the cohomology ring (for instance, the cup product for singular cohomology, or the wedge product for de Rham cohomology). Stiefel-Whitney and Chern classes satisfies the following fundamental properties.
\begin{enumerate}
\item This is called \emph{naturality}, and follows from homotopic invariance of the classifying map: if $g:B'\rightarrow B$ is a continuous map, then
\[
w(g^*E)=g^*w(E)\quad|\quad c(g^*E)=g^*c(E)
\]
\item The \emph{Whitney product formula} follows from the splitting principle \cite[page 66]{HatVBK}:
\begin{equation}\label{eq:product}
w(E\oplus E')=w(E)\cdot w(E')\quad|\quad c(E\oplus E')=c(E)\cdot c(E')
\end{equation}
\item The \emph{normalization property} follows from the very definition: $w_1(\BGLR{1})$ is the generator of $H^1(\BGLR{1};\ZZ_2)$ | $c_1(\BGLC{1})$ is the generator of $H^2(\BGLC{1};\ZZ)$.
\end{enumerate}
\end{rem}

\begin{rem}
As usually happens, definition is not as important as properties, and for characteristic classes this is particularly true: in fact, the above properties can be used to define characteristic classes axiomatically. This approach was first taken by Hirzebruch for Chern classes in \cite[pages 60--63]{HirNTM}, see \cite[End of IV.1.E]{DieHAD}.
\end{rem}

\begin{rem}\label{rem:conj}
The \emph{conjugate bundle} $\bar{E}\rightarrow B$ has the same underlying real vector bundle $E_\RR$, and the conjugate complex multiplication on fibers: $z\cdot_{\bar{E}}v\ug\bar{z}\cdot_E v$. In terms of associated almost complex structure $J$, we have $E=(E_\RR,J)$ and $\bar{E}=(E_\RR,-J)$ \cite[page 123]{KoNFD2}. For the conjugate bundle, Chern classes are given by
\begin{equation}\label{frm:conj}
c_k(\bar{E})=(-1)^k c_k(E)
\end{equation}
This formula comes from the fact that the identity map $\Id:E\rightarrow\bar{E}$ is conjugate linear: $\Id(z\cdot_E v)=\bar{z}\cdot_{\bar{E}}\Id(v)$ \cite[page 167]{MiSChC}.
\end{rem}

\begin{rem}\label{rem:sta-tri}
Denote by $\trivial{n}$ the trivial vector bundle $\RR^n\times B\rightarrow B$ or $\CC^n\times B\rightarrow B$ (which one, will always be clear from the context).
A rank $n$ vector bundle $E$ is \emph{stably trivial} if for a certain $k$ one has $E\oplus\trivial{k}=\trivial{n+k}$. Since the constant map induces the trivial bundle, naturality and product formula in remark \ref{rem:axiom} gives that Stiefel-Whitney and Chern classes are trivial for any stably trivial vector bundle.
This is in particular true for tangent bundles $T(S^n)$ of spheres, because $T(S^n)\oplus\trivial{1}=\trivial{n+1}$, where $\trivial{1}$ is the normal bundle of the embedding $S^n\subset\RR^{n+1}$.
\end{rem}

\subsection{Pontryagin classes}

If you are wondering why $H^*(\BGLR{n};\ZZ)$ doesn't appear in theorem \ref{thm:cohomology}, the reason is that its ring structure is a bit more complicated. In short, $H^*(\BGLR{n};\ZZ)$ contains torsion classes of order 2 which are ``essentially'' Stiefel-Whitney classes - namely, images under the cohomology map $H^i(\BGLR{n};\ZZ_2)\rightarrow H^{i+1}(\BGLR{n};\ZZ_2)$ induced by the short exact sequence
\[
0\rightarrow\ZZ_2\stackrel{2\cdot}{\rightarrow}\ZZ_4\rightarrow\ZZ_2\rightarrow 0
\]
Factoring out these torsion classes, one is left with a free algebra as in theorem \ref{thm:cohomology}:
\begin{equation}\label{frm:pontr}
\frac{H^*(\BGLR{n};\ZZ)}{\text{torsion}}=\ZZ[p_1,\dots,p_k],\qquad p_i\in H^{4i}(\BGLR{n};\ZZ),\qquad n=2k \text{ or } n=2k+1
\end{equation}
For a thorough description of the ring structure of $H^*(\BGLR{n};\ZZ)$, see \cite{BroCBB}.

\begin{dfn}\label{dfn:pontr}
Let $E\rightarrow B$ be a rank $n=2k$ or $n=2k+1$ real vector bundle, $f:B\rightarrow\BGLR{n}$ a classifying map and $p_i$ the generator of $H^{4i}(\BGLR{n};\ZZ)/\text{torsion}$ as in formula \eqref{frm:pontr}. The cohomology classes $p_i(E)\ug f^*(p_i)\in H^{4i}(B;\ZZ)$, for $i=1,\dots,k$, are called the \emph{Pontryagin classes} of $E\rightarrow B$. The \emph{total Pontryagin class} is then $p(E)\ug 1+p_1(E)+\dots+p_k(E)$.
\end{dfn}

Pontryagin classes can be also defined by means of Chern classes. Consider the complexification $E_\CC\ug E\otimes\CC$, and notice that $E_\CC=\overline{E_\CC}$. Then, since conjugation swaps sign for odd-indexed Chern classes by formula \eqref{frm:conj}, the only non-torsion Chern classes are
\begin{equation}\label{eq:pontr}
p_i(E)\ug (-1)^i c_{2i}(E_\CC)\in H^{4i}(B;\ZZ),\qquad i=1,\dots,[n/2]=k
\end{equation}


\begin{rem}\label{rem:sta-tri-pontr}
Since Pontryagin classes are defined by means of Chern classes, remark \ref{rem:sta-tri} still holds: all Pontryagin classes of a stably trivial bundle vanish. In particular, for any $n$ we have $p(T(S^n))=1$, or, in other words, all Pontryagin classes of spheres vanish.
\end{rem}

To close this section, a trivia on \emph{rational} Pontryagin classes of tangent bundles: they are topological invariants, as shown in \cite{NovMFA}, so that one would expect  a topological definition. But such a definition is still missing. However, a combinatorial definition was given in \cite{ThoCCP}.

\subsection{Euler class}

Whenever we have a distinguished orientation on a real vector bundle $E\rightarrow B$, another characteristic class appears: the \emph{Euler class} $e(E)$. In terms of classifying space, the Euler class belongs to the cohomology of $\BG{\SO{n}}=\BGSO{n}$, where $\BGSO{n}$ is the Grassmannian of \emph{oriented} $n$-dimensional vector subspaces of $\RR^{\infty}$. Whenever $n$ is odd, $E$ admits an orientation reversing isomorphism, and this translates to $e(E)=e(-E)=-e(E)$, meaning that $e(E)$ is a torsion class. For this reason, the Euler class is often considered only for even $n=2k$.

The following is the cohomology structure up to torsion of $\BGSO{2k}$:
\begin{equation}\label{frm:euler}
\frac{H^*(\BGSO{2k};\ZZ)}{\text{torsion}}=\ZZ[p_1,\dots,p_{k-1},e],\quad p_i\in H^{4i}(\BGSO{2k};\ZZ),\quad e\in H^{2k}(\BGSO{2k};\ZZ)
\end{equation}

If $B$ is an oriented smooth manifold of dimension $d$, with Euler characteristic
\[
\chi(B)\ug\sum_{i=0}^d(-1)^i\rank{H_i(B;\ZZ)}
\]
and $[B]\in H_d(B;\ZZ)$ the fundamental homology class of $B$, we have
\[
\chi(B)=\langle e(T(B),[B]\rangle
\]
where the brackets denote the usual Kronecker pairing. This is the reason of the name for the Euler characteristic class.

When $E\rightarrow B$ is complex of rank $n$, the underlying real vector bundle $E_\RR$ has rank $2n$, and $e(E_\RR)\in H^{2n}(B;\ZZ)$ can be identified with the top Chern class of $E$:
\begin{equation}\label{eq:topchern}
e(E_\RR)=c_{n}(E)
\end{equation}

\section{Ruling out $S^{4k}$}

The same argument showing that $S^4$ does not admit an almost complex structure holds for every $S^{4k}$, and this latter is somehow clearer to state. Therefore, in this section we consider $S^{4k}$, and prove that it doesn't admit any almost complex structure.
The core of the proof is the fact that an almost complex structure on $S^{4k}$ would imply the last Pontryagin class to be twice the Euler class.

\begin{lem}\label{lem:eulpon}
If there is an almost complex structure on $S^{4k}$, then
\[
(-1)^k p_k(T(S^{4k}))=2e(T(S^{4k}))
\]
\end{lem}

\begin{proof}
Assume that an almost complex structure $J$ on $S^{4k}$ exists. Then each tangent space is a complex vector space \cite[page 114]{KoNFD2}, and $(T(S^{4k}),J)$ is a complex vector bundle that here we denote briefly by $T$. Since $S^{4k}$ has non-trivial cohomology only in dimension $4k$, all Pontryagin and Chern classes vanish except $p_k(T(S^{4k}))$ and $c_{2k}(T)$. Moreover, extending $J$ to a complex linear endomorphism of $T(S^4)\otimes\CC$, we obtain the $\pm i$-eigenspaces splitting $T(S^4)\otimes\CC=T\oplus\bar{T}$ \cite[IX, proposition 1.5]{KoNFD2}. Thus
\[
\begin{split}
1+(-1)^k p_k(T(S^{4k}))\stackrel{\eqref{eq:pontr}}{=}c(T(S^{4k})\otimes\CC)&=c(T\oplus \bar{T})\stackrel{\eqref{eq:product}}{=}c(T)\cdot c(\bar{T)})\\
&=(1+c_{2k}(T))\cdot(1+c_{2k}(T))=1+2c_{2k}(T)\stackrel{\eqref{eq:topchern}}{=}1+2e(T(S^{4k}))
\end{split}
\]
and this proves the lemma.
\end{proof}

\begin{thm}\label{thm:hirzebruchs4k}
The spheres $S^{4k}$ for $k\geq 1$ do not admit any almost complex structure.
\end{thm}

\begin{proof}
If $S^{4k}$ admits an almost complex structure, we get $(-1)^k p_k(T(S^{4k}))=2e(T(S^{4k}))$ from lemma \ref{lem:eulpon}. But $e(T(S^{4k}))$ is non-trivial on even-dimensional spheres, because $\chi=2\neq 0$, and this contradicts remark \ref{rem:sta-tri-pontr} stating that all Pontryagin classes of spheres are trivial. More explicitly, denoting by $[S^{4k}]$ the fundamental class of $S^{4k}$  we have
\[
\langle p_k(T(S^{4k})),[S^{4k}]\rangle=(-1)^k\langle 2e(T(S^{4k}),[S^{4k}]\rangle=(-1)^k 2\chi(S^{4k})=(-1)^k 4\neq 0
\]
\end{proof}

\subsection{Ruling out $S^{4k}$ via Hirzebruch Signature Theorem}

Another nice way to prove theorem \ref{thm:hirzebruchs4k} is by using the Hirzebruch Signature Theorem. In short, lemma \ref{lem:eulpon} contradicts the fact that the signature of $S^{4k}$ must be zero. For completeness, we give here a brief sketch of this argument.

Let $M^{4k}$ be a connected, compact, smooth manifold of dimension $4k$. The singular cohomology cup product on $H^{2k}(M;\RR)$ and the Poincaré duality induce a symmetric bilinear form:
\begin{align}
\beta:H^{2k}(M;\RR)\times H^{2k}(M;\RR)&\rightarrow H^{4k}(M;\RR)\cong H^0(M;\RR)=\RR\\
(x,y) &\mapsto x\cup y
\end{align}
By Poincaré duality, $\beta$ is non-degenerate. The \emph{signature} $\sigma(M)$ of $M$ is the signature of $\beta$.

In the following, we briefly describe a relation between the signature of $M$ and the Pontryagin classes of $TM$ which is known as the \emph{Hirzebruch Signature Theorem}. For a more systematic discussion see \cite[Chapter I]{HirTMA}.

Consider the power series of
\[
Q(z) = \frac{\sqrt{z}}{\tanh\sqrt{z}} = 1+\sum_{k=1}^\infty (-1)^{k-1} \frac{2^{2k}}{(2k)!} B_k z^k
\]
where $B_k$ are the \emph{Bernoulli numbers}. In particular, $B_k>0$ and $\neq \frac{1}{2}$ for all $k$. Let $p_1,\dots,p_m$ be variables for $m \in \NN$. By the fundamental theorem of symmetric polynomials there are variables $\beta_1,\dots,\beta_m$ such that
\[
1+p_1z+p_2z^2+\ldots+p_mz^m = \prod_{i=1}^m (1+\beta_i z),
\]
i.e.\ we may consider $p_i$ as elementary symmetric polynomials in $\beta_1,\dots,\beta_m$. Now consider the coefficient of $z^k$ in the product $\prod_{i=1}^mQ(\beta_i z)$. Clearly it is a symmetric polynomial in $\beta_1,\ldots,\beta_m$, homogeneous of weight $k$. Therefore it can be expressed as a polynomial $L_k^m(p_1,\dots,p_k)$ in a unique way and moreover it does not depend on $m$ for $k\leq m$. Consequently we write $L_k\ug L^m_k$ for $k\leq m$. The coefficient $s_k$ of $p_k$ in $L_k$ is determined by
\[
\sum_{k=0}^\infty s_k  z^k = \frac{1}{2} + \frac{1}{2}\frac{2\sqrt{z}}{\sinh 2\sqrt{z}}
\]
hence (see \cite[page 12]{HirTMA})
\[
s_0 =1,\quad s_k = \frac{2^{2k}(2^{2k-1}-1)}{(2k)!}B_k,\; k\geq 1
\]
The first polynomials $L_k$ are given by
\[
L_1 = \frac{1}{3}p_1, \quad L_2=\frac{1}{45}\left( 7p_2 -p_1^2\right), \quad L_3 = \frac{1}{3^2\cdot 5\cdot 7}\left(62p_3 -13p_2p_1 +2p_1^3\right),\quad \dots
\]

\begin{dfn}
The \emph{$L$-class} of $M^{4k}$ is the cohomology class
\[
L_k(p_1,\dots,p_k) \in H^{4k}(M^{4k};\QQ),
\]
where $p_i = p_i(TM^{4k})$.
\end{dfn}


\begin{thm}\label{thm:signa}\cite[Hirzebruch Signature Theorem]{HirTMA}

Let $M^{4k}$ be a compact and oriented manifold and let $[M] \in H_{4k}(M)$ be the fundamental class of $M$. Then we have
\[
\sigma(M) = \langle L_k(p_1,\dots,p_k), [M]\rangle
\]
In particular $L_k(p_1,\dots,p_k)$ is integral.
\end{thm}

\emph{Alternative proof of theorem \ref{thm:hirzebruchs4k}.}

Since the cohomology of $S^{4k}$ is trivial except in dimensions $0$ and $4k$, all Pontryagin classes of $TS^{4k}$ are zero, except possibly the top class $p_k$, hence $L_k(S^{4k})=s_kp_k$.
Moreover the signature $\sigma(S^{4k})$ vanish, since $H^{2k}(S^{4k};\ZZ)=0$.
Combining lemma \ref{lem:eulpon} with the signature formula in theorem \ref{thm:signa}, we obtain a contradiction as follows
\[
0 =\sigma(S^{4k})=\langle s_kp_k,[S^{4k}]\rangle=(-1)^k 2\langle s_k e(T(S^{4k})),[S^{4k}]\rangle=(-1)^k 2s_k\chi(S^{4k})=(-1)^k 4s_k \neq 0\hfill\qed
\]

\section{K-Theory, Bott Periodicity Theorem and the Chern character}

We start with the definition  of topological $K$-theory. Let $B$ be a compact CW-complex and let $F(B)$ denote the free abelian group of isomorphism classes of complex vector bundles over $B$ (in the following the expression \emph{vector bundle} shall always mean complex vector bundle). We define $K(B)$ as the quotient qroup of $F(B)$ by the subgroup generated by elements of the form $[E\oplus F] - [E] -[F]$, where $E,F$ are vector bundles over $B$ and $[E]$ is the isomorphism class of $E$. It is easy to see that vector bundles $E\to B$ and $F \to B$ represent the same class in $K(B)$ if and only if $E$ and $F$ are \emph{stably isomorphic}, i.e.\ there is an $n \in \NN$ and a trivial bundle $\trivial{n} \to B$ of rank $n$ such that $E\oplus \trivial{n}$ is isomorphic to $F \oplus \trivial{n}$. Finally we will denote with $x_E$ the stable class of $E$ in $K(B)$.

In the following paragraphs we would like to present some properties of $K(B)$ which will be needed in the proof of Theorem \ref{thm:main}.

\subsection{Ring structure of $\mathbf{K(B)}$}

By construction $K(B)$ is an abelian group with respect to direct sums of vector bundles. The zero element is the trivial vector bundle of rank $0$. In addition $K(B)$ has also a multiplication coming from tensor product of vector bundles. More precisely, for elements in $K(B)$ represented by vector bundles $E$ and $F$, their product in $K(B)$ will be represented by $E\otimes F$. We extend this definition distributively to all elements of $K(B)$ and we obtain a well-defined multiplication on $K(B)$. This makes $K(B)$ into a commutative ring with unit (where the unit is represented by the trivial line bundle $\trivial{1} \to B$).

If $B'$ is another compact CW-complex and $f \colon B \to B'$ a continuous map, then it induces a ring homomorphism $K(f):K(B') \to K(B)$ by pulling back vector bundles. Hence $K$ can be regarded as a contravariant functor from the category of compact CW-complexes into the category of commutative rings.

\subsection{Bott periodicity in $K$-theory}

Let $B$ and $B'$ be compact CW-complexes. There is an \emph{external product} $\mu:K(B)\otimes K(B') \to K(B\times B')$, where $K(B)\otimes K(B')$ is the tensor product of rings, which is a ring again. If $a \otimes b \in K(B) \otimes K(B')$ then 
\[
\mu(a\otimes b) = p_B^*(a)p_{B'}^*(b)
\]
where $p_B$ and $p_{B'}$ are the projections of $B\times B'$ onto $B$ and $B'$ respectively. We extend $\mu$ linearly to all elements of $K(B) \otimes K(B')$. Taking $B'=S^2$ we obtain

\begin{thm}[Bott periodicity, cf. Corollary 2.2.3 in \cite{AtiKTh}]\label{thm:BottPeriodicity}
The homomorphism 
\[
\mu \colon K(B) \otimes K(S^2) \to K(B \times S^2)
\] 
is an isomorphism of rings.
\end{thm}

\subsection{The Chern character}

Let $L \to B$  and $L' \to B$ be two complex line bundles. The first Chern class is an isomorphism from the group of isomorphism classes of line bundles over $B$ \footnote{where the multiplication is defined by the tensor product of line bundles} to $H^2(B;\ZZ)$, in particular $c_1(L\otimes L') = c_1(L) + c_1(L')$. In the same way, we would like to have a natural ring homomorphism from $K$-theory to (ordinary) cohomology. The product formula for line bundles suggest that we take the exponential of the first Chern class to preserve the multiplication of $K(B)$ into cohomology. Hence we define the \emph{Chern character of a line bundle} to be
\[
\mathop{ch}(L)=e^{c_1(L)} = 1 + c_1(L) +\frac{c_1(L)^2}{2!} + \frac{c_1(L)^3}{3!}+ \ldots \in H^*(B;\QQ)
\]
and since $B$ was assumed to be compact, the formal sum will be finite. Clearly we obtain
\[
\mathop{ch}(L\otimes L')=e^{c_1(L\otimes L')} = e^{c_1(L) + c_1(L')} = e^{c_1(L)}e^{c_1(L)} = \mathop{ch}(L)\cup \mathop{ch}(L').
\]
We extend this definition linearly to direct sums of line bundles, i.e. if $E=L_1\oplus\ldots\oplus L_n$, then
\[
\mathop{ch}(E) = \sum_i \mathop{ch}(L_i) = n + (t_1+\ldots +t_n) + \ldots + (t_1^k+\ldots+t_n^k)/k! + \ldots
\]
where $t_i = c_1(L_i)$. In this notation, the total Chern class is given by $c(E)=(1+t_1)(1+t_2)\cdot\ldots\cdots(1+t_n) = 1 + \sigma_1 + \ldots + \sigma_n$, where $\sigma_j = c_j(E)$ is the $j$-th elementary symmetric polynomial in the $t_i$'s. Moreover, if $\nu_k$ denotes the $k$-th Newton polynomial (cf.\ \cite{MiSChC}) it satisfies the relation
\[
t_1^k+\ldots+t_n^k = \nu_k(\sigma_1,\ldots,\sigma_k) =\nu_k(c_1(E),\ldots,c_k(E)).
\]
Hence the \emph{Chern character of $E$} becomes
\begin{equation}
\mathop{ch}(E) = \dim E + \sum_{k> 0}\frac{1}{k!}\nu_k(c_1(E),\ldots,c_k(E)).
\end{equation}
Since the right hand side depends only on the vector bundle $E$, we take this as a definition for the Chern character for an arbitrary vector bundle over $B$. Using the \emph{splitting principle} it is easy to show 
\begin{prp}[cf. Proposition 4.2 in \cite{HatVBK}]\label{prop:chSumsAndProducts}
For complex vector bundles $E$ and $E'$ over $B$ we have 
\begin{enumerate}
\item $\mathop{ch}(E\oplus E') = \mathop{ch}(E) + \mathop{ch}(E')$,
\item $\mathop{ch}(E\otimes E') = \mathop{ch}(E) \cup \mathop{ch}(E')$.
\end{enumerate}
\end{prp}

Note that for Chern classes we have the identity $c(E\oplus\trivial n) = c(E)$, hence $c_k(E \oplus\trivial n) = c_k(E)$ for all $k$. It follows that the Chern character is well-defined on the stable class $x_E \in K(B)$ of $E$ and therefore it descends to a map $\mathop{\mathrm{ch}}\colon K(B) \to H^*(B;\QQ)$. From Proposition \ref{prop:chSumsAndProducts} we obtain

\begin{prp}\label{prop:chRingHom}
The Chern character $\mathop{\mathrm{ch}}\colon K(B) \to H^*(B;\QQ)$ is a ring homomorhism. 
\end{prp}

\begin{rem}\label{rem:naturality}
Firstly, let $\ch{B}$ denote the image of $\mathop{\mathrm{ch}}\colon K(B) \to H^*(B;\QQ)$. We call $\ch{B}$ \emph{integral} if $\ch{B} \subset H^*(B;\ZZ)$. Secondly, let $B'$ be another compact CW-complex and $f:B \to B'$ a continuous map, then by the naturality of Chern classes we obtain the naturality of the Chern character, i.e.\ for $y \in K(B')$ we have $f^*(\ch{y}) = \ch{K(f)(y)}$.
\end{rem}

\section{Proof of Theorem \ref{thm:main} }
We start with an observation that $\ch{S^2}$ is integral. For if $E \to S^2$ is a complex vector bundle then $\ch{x_E}= n + c_1(E) \in H^*(S^2;\ZZ)$. Moreover we have
\begin{prp}\label{prop:ch(S2n)isIntegral}
$\ch{S^{2n}}$ is integral.
\end{prp}

\begin{proof}
Let $f \colon S^2 \times\ldots\times S^{2} \to S^{2n}$ be a map of degree $1$, where the cartesian product is taken $n$ times. Then the induced map on cohomology $f^* \colon H^*(S^{2n}
) \to H^*(S^2\times\ldots\times S^2)$ is an monomorphism (if the coefficient group in $H^*$ is not specified, it is assumed to be the rationals) . Using the Künneth-formula we also have an isomorphism of rings $H^*(S^2\times\ldots\times S^2) \cong H^*(S^2)\otimes\ldots\otimes H^*(S^2)$, where the tensor product is taken $n$ times.

Moreover by the Bott periodicity theorem, cf. Theorem \ref{thm:BottPeriodicity}, the map $\mu$ induces inductively an isomorphism
\[
K(S^2\times\ldots\times S^n) \to K(S^2)\otimes\ldots\otimes K(S^2).
\]
and the following diagramm commutes
\begin{center}
\begin{tikzpicture}[scale=2]
\node (A) at (0,1) {$K(S^2\times\ldots\times S^2)$};
\node (B) at (4,1) {$K(S^2)\otimes\ldots\otimes K(S^2)$};
\node (C) at (0,0) {$H^*(S^2\times\ldots\times S^2)$};
\node (D) at (4,0) {$H^*(S^2)\otimes\ldots\otimes H^*(S^2)$.};
\path[->,font=\scriptsize,>=angle 90]
(A) edge node[above]{Bott periodicity} (B)
(A) edge node[right]{$\mathop{\mathrm{ch}}$} (C)
(B) edge node[right]{$\mathop{\mathrm{ch}}\otimes\ldots\otimes \mathop{\mathrm{ch}}$} (D)
(C) edge node[above]{Künneth isomorphism} (D);
\end{tikzpicture}
\end{center}
Since $\ch{S^2}$ is integral, it follows that the image of $\mathop{\mathrm{ch}}\otimes\ldots\otimes \mathop{\mathrm{ch}}$ lies in $H^*(S^2;\ZZ)\otimes\ldots\otimes H^*(S^2;\ZZ)$. Hence $\ch{S^2\times\ldots\times S^2}$ has to be integral.

The naturality of Chern character (cf. Remark \ref{rem:naturality}) implies that the diagramm
\begin{center}
\begin{tikzpicture}[scale=2]
\node (A) at (0,1) {$K(S^{2n})$};
\node (B) at (2,1) {$K(S^2\times\ldots\times S^2)$};
\node (C) at (0,0) {$H^*(S^{2n})$};
\node (D) at (2,0) {$H^*(S^2\times\ldots\times S^2)$.};
\path[->,font=\scriptsize,>=angle 90]
(A) edge node[above]{$K(f)$} (B)
(A) edge node[right]{$\mathop{\mathrm{ch}}$} (C)
(B) edge node[right]{$\mathop{\mathrm{ch}}$} (D)
(C) edge node[above]{$f^*$} (D);
\end{tikzpicture}
\end{center}
also commutes. From above we have that $\ch{S^2\times\ldots\times S^2}$ is integral and since $f^*$ is an isomorphism we conclude that $\ch{S^{2n}}$ has to be integral too.
\end{proof}
Suppose now, there is a complex vector bundle $E$ of rank $n$ such that $E_\RR = TS^{2n}$. Then, due to the cohomology of $S^{2n}$ all Chern classes have to be zero except possibly $c_n(E)$. From the recursive formula for Newton polynoms, cf.\cite[p.195]{MiSChC}, it follows that only $\nu_n$ is not zero and is given by $\nu_n = \pm n c_n(E)$. Hence $\ch{x_E}=n \pm \frac{c_n(E)}{((n-1)!)}$ and therefore 
\[
\frac{c_n(E)}{(n-1)!} \in H^{2n}(S^{2n};\ZZ)\cong\ZZ.
\]
Since $\langle c_n(E),[S^{2n}]\rangle = \langle e(E_\RR)),[S^{2n}]\rangle =  \chi(S^{2n})= 2$, we conclude that $(n-1)!$ has to divide $2$, hence $n\leq 3$.

This completes the proof of Therem \ref{thm:main}.

\begin{rem}
The proof of Theorem \ref{thm:main} uses only the homotopy type of $S^{2n}$. By the generalized Poincare conjecture every manifold homotopy equivalent to $S^{2n}$ is homeomorphic to $S^{2n}$. But since Milnor \cite{MilMHS} it is known that a manifold can be homeomorphic but not diffeomorphic to the standard sphere $S^{2n}\subset\RR^{2n+1}$. These manifolds are called \emph{exotic spheres}. Thus Theorem \ref{thm:main} is also true if $S^{2n}$ is replaced by an exotic sphere.
\end{rem}

\end{document}